\documentclass[12pt]{amsart}%

\usepackage{amssymb}
\usepackage{amsmath}
\usepackage{amsthm}
\usepackage{amscd,mdwlist}
\usepackage[margin=1in]{geometry}
\usepackage{hyperref}
\usepackage{color}
\usepackage{cite}
\usepackage{bbm}
    
\theoremstyle{plain}
\newtheorem{theorem}{Theorem}[section]
\newtheorem{proposition}{Proposition}[section]
\newtheorem{corollary}{Corollary}[section]
\newtheorem{lemma}{Lemma}[section]
\newtheorem{definition}{Definition}[section]

\theoremstyle{remark}
\newtheorem{remark}{Remark}[section]

\newtheorem{assumption}{Assumption}[section]

\DeclareMathOperator{\supp}{supp}

\DeclareMathOperator{\tr}{Tr}
\DeclareMathOperator{\rank}{rank}
\DeclareMathOperator{\diverg}{div}

\DeclareMathOperator{\loc}{loc}

\DeclareMathOperator{\E}{\mathcal{E}}

\begin{document}

\title{Sobolev spaces and calculus of variations on fractals}

\author{Michael Hinz$^1$, Dorina Koch$^2$, Melissa Meinert$^3$}
\address{$^1$ Fakult\"at f\"ur Mathematik, Universit\"at Bielefeld, Postfach 100131, 33501 Bielefeld, Germany}
\email{mhinz@math.uni-bielefeld.de}
\address{$^2$ Fakult\"at f\"ur Mathematik, Universit\"at Bielefeld, Postfach 100131, 33501 Bielefeld, Germany}
\email{dorina-koch@t-online.de}
\address{$^3$ Fakult\"at f\"ur Mathematik, Universit\"at Bielefeld, Postfach 100131, 33501 Bielefeld, Germany}
\email{mmeinert@math.uni-bielefeld.de}
\thanks{$^1$, $^3$ Research supported in part by the DFG IRTG 2235: 'Searching for the regular in the irregular: Analysis of singular and random systems'.}
%

\begin{abstract}
The present note contains a review of $p$-energies and Sobolev spaces on metric measure spaces that carry a strongly local regular Dirichlet form. These Sobolev spaces are then used to generalize some basic results from the calculus of variations, such as the existence of minimizers for convex functionals and certain constrained mimimization problems. This applies to a number of non-classical situations such as degenerate diffusions, superpositions of diffusions and diffusions on fractals or on products of fractals.
\tableofcontents
\end{abstract}

\keywords{Dirichlet forms, fractals, $p$-energies, Sobolev spaces, convex problems, constrained problems}
\subjclass[2010]{28A80, 35J20, 35J92, 46B10, 46E35, 47A07, 49J40, 49J45}

\maketitle

\section{Introduction} 

By now linear elliptic or parabolic partial differential equations on various fractal spaces have successfully been studied for quite some time. Highly readable introductions are provided in \cite{Ba98, Ki01} and \cite{Str06}. Less is known about semilinear equations, and there is very little existing research connected to quasilinear problems, in particular, if equations explicitely involve gradients. One of the easiest quasilinear problem is the $p$-Laplace equation $\diverg (|\nabla u|^{p-2}\nabla u)=0$, and a weak formulation of this equation involves the nonlinear $p$-energy functional $\int |\nabla u|^pdx$. While quadratic functionals, more precisely, Dirichlet forms, are a central tool in the analysis on fractals, not too much is known about nonlinear functionals. 
A similar statement could be made about Sobolev spaces: Substitutes of the spaces $W^{1,2}(\Omega)$ or $W_0^{1,2}(\Omega)$ arise naturally in the linear theory as domains of Dirichlet forms, but more general Sobolev spaces on fractals are barely discussed. A closely related area in the study of partial differential equations is the calculus of variations, and some of its basic methods are equally robust as Dirichlet form theory. One textbook example is the existence of minimizers for the $p$-energy in a certain domain, in other words, the existence of weak solutions to the Dirichlet problem for the $p$-Laplacian. But of course variational methods are much more versatile.

Here we pursue two aims. First, we would like to give quick account on $p$-energies and $(1,p)$-Sobolev spaces for fractals that carry a local regular Dirichlet form. This means, our starting point is the linear theory for $p=2$, and we would like to take this as a base to define $p$-energies and Sobolev spaces. The naive idea is to mimick the classical definitions, and clearly this can work only under certain assumptions. We assume that the volume measure is 
energy dominant (see Section \ref{S:Defs}) and that there is a sufficiently large pool $\mathcal{A}$ of functions with bounded gradients (Assumption \ref{A:basic}). This excludes many interesting examples, but it includes others, see Remark \ref{R:inex} (ii) below. In particular, we can study the Sierpinski gasket endowed with the standard energy form and the Kusuoka measure, which we regard as a prototype example. The discussion of Sobolev spaces and calculus of variations is naturally connected to concepts of measurable bundles and first order derivations, see for instance \cite[Section 4.3]{CG03}. One consequence of this connection is the reflexivity of $L^p$-spaces of vector fields, it leads to the \emph{reflexivity} of related \emph{Sobolev spaces}. This fact is used for the existence of minimizers and was not pointed out in our former study in \cite{HRT13}. Another related issue is the closability of $p$-energies respectively completeness of Sobolev spaces. In \cite[Section 6]{HRT13} we defined analogs $H_0^{1,p}(X,m)$ of the classical Sobolev spaces $W_0^{1,p}(\Omega)$ for $p\geq 2$, for which completeness is rather trivial. Here we advance a tiny bit further and show that under additional assumptions (Assumption \ref{A:AL}) also the spaces $H_0^{1,p}(X,m)$ for $1<p< 2$ will be complete. Under this assumption one can, as a by-product, also define \emph{distributional gradients} and analogs $W^{1,p}(X,m)$ of the classical Sobolev spaces $W^{1,p}(\Omega)$. Also these refinements apply to the Sierpinski gasket with Kusuoka measure.

In a sense our exposition of Sobolev spaces is quite similar to \cite[Section 4.3]{CG03}. However, there the authors considered $p$-energies in $L^2$-spaces and studied their contraction properties, what is strongly connected to nonlinear Markovian semigroups and monotone operator methods. Here we consider $p$-energies in $L^p$-spaces. 

Our second aim is to review a classical result (\cite[Theorem 4.3.1]{Jost98}, see also \cite{Da08}) about the existence of minimizers for convex functionals in the present setup, Theorem \ref{T:exmin}. The payoff is that one can now discuss such problems for degenerate energy functionals with 'varying tangent space dimensions', for superpositions of diffusions, for certain diffusions on fractals or on their products see Section \ref{S:Ex1}. The functionals do not have to be isotropic. We also discuss some constrained models. 
 
The present note is intended to be expository, we do not claim any substantial news. The experienced reader could probably easily compile our results from facts about first order derivations and measurable bundles as for instance studied in \cite{CS03, CS09, Eb99, Gi15, HRT13, IRT12, W00} and ideas similar to \cite[Section 4.3]{CG03}. However, we are not aware of any earlier exposition of the present material, and wrapping things up in this note will hopefully save the interested reader from having to collect all the necessary fragments from different sources. 

In Section \ref{S:Defs} we sketch our setup and assumptions, provide definitions, discuss our approach and others. 
Section \ref{S:bundle} contains a brief account on first order derivations and measurable bundles, and it points out the claimed reflexivity. Some examples are provided in Section \ref{S:Ex1}, convex problems and constrained problems are discussed in Sections \ref{S:Exmin} and \ref{S:constrained}, followed by brief examples in Section \ref{E:Ex2}. The proofs of reflexivity and closability and the definition of distributional gradients are shifted to an appendix. 

For quantities  $(f,g)\mapsto \mathcal{Q}(f,g)$ depending on two arguments $f,g$ in a symmetric way we use the notation $\mathcal{Q}(f):=\mathcal{Q}(f,f)$.

\section{Dirichlet forms, $p$-energies and Sobolev spaces $H_0^{1,p}(X,m)$}\label{S:Defs}

Throughout this note $(X,d)$ is assumed to be a locally compact separable metric space and $m$ a nonnegative Radon measure on $X$ such that $m(U)>0$ for any nonempty open set $U \subset X$. 

Recall that a pair $(\E, \mathcal{D}(\E))$ is called a symmetric \emph{Dirichlet form} on $L^2(X, m)$, \cite{BH91, ChF12,FOT94}, if $\mathcal{D}(\mathcal{E})$ is a dense subspace of $L^2(X,m)$ and $\mathcal{E}:\mathcal{D}(\mathcal{E})\times\mathcal{D}(\mathcal{E})\to\mathbb{R}$ is a symmetric nonnegative definite bilinear form which is \emph{closed}, i.e. such that $\mathcal{D}(\mathcal{E})$ with the scalar product 
\[\E_1 (f,g) := \E(f,g) + \langle f,g \rangle_{L^2(X, m)}\] 
is a Hilbert space, and satisfies the \emph{Markov property}, which says that 
\[\text{$f\in \mathcal{D}(\mathcal{E})$ implies $(0 \vee f)\wedge 1 \in \mathcal{D}(\mathcal{E})$ and $\E((0 \vee f) \wedge 1) \leq \E(f)$.}\]

A subset of $C_c(X) \cap \mathcal{D}(\mathcal{E})$ is called a \emph{core} of the Dirichlet form $(\E, \mathcal{D}(\mathcal{E}))$ on $L^2(X, m)$ if it is both uniformly dense in the space of compactly supported continuous functions $C_c(X)$ and dense in $(\E_1,\mathcal{D}(\mathcal{E}))$. A Dirichlet form $(\E, \mathcal{D}(\mathcal{E}))$ on $L^2(X, m)$ is called \emph{regular} if it possesses a \emph{core}. A regular Dirichlet form is called \emph{strongly local} if $\E(f,g) =0$ for any $f,g \in C_c(X) \cap \mathcal{D}(\mathcal{E})$ such that $g$ is constant on a neighborhood of $\supp f$, \cite[Section 3.2]{FOT94}. By the Markov property $C_c(X) \cap \mathcal{D}(\mathcal{E})$ is an algebra, see \cite[Corollary I.3.3.2]{BH91}. Similarly as in \cite{BH91} we say that a regular Dirichlet form $(\mathcal{E},\mathcal{D}(\mathcal{E}))$ \emph{admits a carr\'e du champ} if for any $f,g\in C_c(X)\cap \mathcal{D}(\mathcal{E})$ there exists a function $\Gamma(f,g)\in L^1(X,m)$ such that for any $h\in C_c(X)\cap \mathcal{D}(\mathcal{E})$ we have 
\begin{equation}\label{E:carre}
\frac12\{\mathcal{E}(fh,g)+\mathcal{E}(gh,f)-\mathcal{E}(fg,h)\}=\int_Xh\Gamma(f,g)dm.
\end{equation}
This is the same as to say that \emph{the Dirichlet form admits energy densities with respect to $m$} or to say that \emph{the measure $m$ is energy dominant for $(\mathcal{E},\mathcal{D}(\mathcal{E}))$}, \cite{Hino10, Hino13b}.

We define $p$-energies and Sobolev spaces $H^{1,p}_0(X,m)$, $1\leq p<+\infty$, associated with a given regular Dirichlet form $(\mathcal{E},\mathcal{D}(\mathcal{E}))$ on $L^2(X,m)$. This partially generalizes former definitions in \cite[Section 6]{HRT13}, which covered the cases $2\leq p<+\infty$. To keep the exposition simple, we make the following \emph{standing assumption}:

\begin{assumption}\label{A:basic}
The Dirichlet form $(\mathcal{E},\mathcal{D}(\mathcal{E}))$ on $L^2(X,m)$ is strongly local and admits a carr\'e du champ, $m(X)<+\infty$, and $\mathcal{A}$ is an algebra and a core for $(\mathcal{E},\mathcal{D}(\mathcal{E}))$ such that 
\begin{equation}\label{E:basic}
\Gamma(f,g):=\frac{d\Gamma(f,g)}{dm}\in L^\infty(X,m)\quad \text{for all $f,g\in \mathcal{A}$.}
\end{equation}
\end{assumption}

Under Assumption \ref{A:basic} we can define associated \emph{$p$-energies}, $1\leq p<+\infty$, by
\[\mathcal{E}^{(p)}(f):=\int_X\Gamma(f)^{p/2}dm,\quad f\in \mathcal{A}.\]
We wish to extend $\mathcal{E}^{(p)}$ to a subspace of $L^p(X,m)$ consisting of all functions $f$ for which 
$\mathcal{E}^{(p)}(f)$ can be defined as a finite quantity, and we wish this pool of functions to become a
Banach space. The functional $(\mathcal{E}^{(p)},\mathcal{A})$ is said to be \emph{closable in $L^p(X,m)$} if for any sequence $(f_n)_n\subset \mathcal{A}$ that is Cauchy in the seminorm $\mathcal{E}^{(p)}(\cdot)^{1/p}$ and such that $\lim_n f_n=0$ in $L^p(X,m)$ we have $\lim_n \mathcal{E}^{(p)}(f_n)=0$.

For $p\geq 2$ closability is easily seen, for $1<p<+\infty$ we make an additional assumption, it implies a 'distributional' integration by parts identity, see Appendix \ref{S:closable}. Let $(\mathcal{L},\mathcal{D}(\mathcal{L}))$ denote the generator of the Dirichlet form $(\mathcal{E},\mathcal{D}(\mathcal{E}))$ i.e. the unique non-positive definite self-adjoint operator such that 
\begin{equation}\label{E:generator}
\mathcal{E}(f,g)=-\left\langle \mathcal{L}f,g\right\rangle_{L^2(X,m)}
\end{equation}
for all $f\in\mathcal{D}(\mathcal{L})$ and $g\in\mathcal{D}(\mathcal{E})$.

\begin{assumption}\label{A:AL}
There is a space $\mathcal{A}_{\mathcal{L}}\subset \mathcal{D}(\mathcal{L})$, dense in $\mathcal{D}(\mathcal{E})$ and such that for any $f\in\mathcal{A}_{\mathcal{L}}$ we have $\Gamma(f)\in L^\infty(X,m)$ and $\mathcal{L}f\in L^\infty(X,m)$.
\end{assumption}

A proof of the next statement can be found in Appendix \ref{S:closable}. 

\begin{theorem}\label{T:closable}
The functional $(\mathcal{E}^{(p)},\mathcal{A})$ is closable in $L^p(X,m)$,  $2\leq p<+\infty$. If Assumption \ref{A:AL} is satisfied, then it is also closable in $L^p(X,m)$, $1<p<2$.
\end{theorem}

Now suppose that $2\leq p<+\infty$ or that Assumption 
\ref{A:AL} is satisfied. Then, if $f\in L^p(X,m)$ is such that there exists a sequence $(f_n)_n\subset \mathcal{A}$, Cauchy in the seminorm $\mathcal{E}^{(p)}(\cdot)^{1/p}$ and convergent to $f$ in $L^p(X,m)$, we define
\[\mathcal{E}^{(p)}(f):=\lim_n\mathcal{E}^{(p)}(f_n),\]
which by Theorem \ref{T:closable} is a correct definition. We denote the vector space of all such $f\in L^p(X,m)$ by $H_0^{1,p}(X,m)$. The spaces $H_0^{1,p}(X,m)$ are Banach with norms
\[\left\|f\right\|_{H^{1,p}_0(X,m)}=\left\|f\right\|_{L^p(X,m)}+\mathcal{E}(f)^{1/p},\quad f\in H_0^{1,p}(X,m).\]
The functional $(\mathcal{E}^{(p)}, H_0^{1,p}(X,m))$ is called the \emph{closure} (or \emph{smallest closed extension}) of $(\mathcal{E}^{(p)},\mathcal{A})$ in $L^p(X,m)$. Formula (\ref{E:basic}) remains valid for all $f\in H_0^{1,p}(X,m)$, their energy densities $\Gamma(f)$ can be defined by approximation.

\begin{definition}
To the spaces $H_0^{1,p}(X,m)$, $1\leq p<+\infty$, we refer as \emph{Sobolev spaces}. Given an open set $\Omega\subset X$ we define the \emph{Sobolev spaces} $H_0^{1,p}(\Omega,m)$ \emph{on $\Omega$} as the completion in $H_0^{1,p}(X,m)$ of all elements of $\mathcal{A}$ supported in $\Omega$, respectively.
\end{definition}

These spaces are obvious generalizations of the classical Sobolev $W^{1,p}_0(\Omega)$ spaces over bounded domains $\Omega$ in Euclidean spaces, \cite{AF03, Ma11}. For $p=2$ the Sobolev spaces are Hilbert, and $H^{1,2}_0(X,m)=\mathcal{D}(\mathcal{E})$ in the sense of equivalently normed Hilbert spaces.

\begin{remark}\label{R:inex}\mbox{}
\begin{enumerate}
\item[(i)] For the definitions of $p$-energies and Sobolev spaces as above the algebra $\mathcal{A}$ and the measure $m$ are part of the setup. Therefore the definitions depend on the choice of these items, at least a priori. It would be interesting to find out more about the possible equivalence of definitions for different $\mathcal{A}$ and $m$. Valuable hints might be found in \cite{Schm17}.
\item[(ii)] Strong locality of the form and finiteness of the measure in Assumption \ref{A:basic} may not be too restrictive, but the assumption of a carr\'e du champ excludes many interesting examples, such as for instance the standard self-similar Dirichlet forms on the classical Sierpinski gasket and carpet, considered with the standard normalized self-similar Hausdorff measure, respectively. See for instance \cite{BBST99, Hino03, Hino05}. This is a clear disadvantage. However, for a given regular Dirichlet form one can always find a finite measure $m$ and an algebra $\mathcal{A}$, \cite[Lemma 2.1]{HKT15}, \cite[Remark 4.1]{HT18+}, such that after a standard change of measure procedure we arrive at a Dirichlet form satisfying Assumption \ref{A:basic}. For change of measure results see \cite[Corollary 5.2.10]{ChF12}, \cite[Section 6.2, p. 275]{FOT94}, \cite{FLJ91, FST91, KN91}, or \cite{H16, HRT13}. This clearly alters the metric measure space under consideration, but it provides a rich class of examples of fractal spaces that can be analyzed. The study of fractals with energy dominant measures,\cite{Ku89, T08}, is also referred to as \emph{measurable Riemannian geometries}, see \cite{Ka12, Ki08}, and also the related studies \cite{Hino13b, KZh12}. The prototype of these examples is the classical Sierpinski gasket with standard energy form and Kusuoka measure. We would finally like to point out that in the case of resistance forms a considerable amount of theory can be developed in a measure-free context, this has been done in \cite{Ki03}.
\item[(iii)] Assumption \ref{A:AL} is also satisfied for many interesting examples, for instance to the Sierpinski gasket endowed with the standard energy form and the Kusuoka measure, see Section \ref{S:Ex1}. This example generalizes to more general finitely ramified fractals, \cite[Section 8]{T08}. Secondly, if (in addition to Assumption \ref{A:basic}) the Markovian semigroup uniquely on $L^2(X,m)$ associated with $(\mathcal{E},\mathcal{D}(\mathcal{E}))$ is a Feller semigroup (i.e. a positivity preserving and strongly continuous contraction semigroup on the space $C_0(X)$ of continuous functions vanishing at infinity) and $\mathcal{A}_{\mathcal{L}}$ is a dense subalgebra of the domain of the Feller generator, then Assumption \ref{A:AL} is satisfied. This can be seen as in \cite[Lemma 2.8]{BBKT10}.  Finally, Assumption \ref{A:AL} also holds if (in addition to the other assumptions) $(X,m,\Gamma)$ is a Diffusion Markov triple in the sense of \cite[Definition 3.1.8]{BGL14}.
\end{enumerate}
\end{remark}

\begin{remark}\label{R:environment}
The definition of $p$-energies and $(1,p)$-Sobolev spaces based on a given Dirichlet form differs from other well established approaches: 

Sobolev spaces on metric measure spaces via mean-value type inequalities have been proposed in \cite{Haj96}, see also \cite[Section 5.4]{Hei01}. For Sobolev spaces on metric measure spaces via rectifiable curves and upper gradients see \cite{BB11, BBSh13, HKShT15, KinM02, Sh00, Sh03} or \cite[Section 7]{Hei01}, an equivalent approach is provided in \cite{Ch99}. By a lack of rectifiable curves this upper gradient approach does not apply to fractals. 

Originating in Dirichlet form theory for homogeneous spaces \cite{BM95, MaMos99}, a sort of axiomatic approach to nonlinear energy forms was suggested in \cite{Ca03, Ca03b, Mos05}, it is related to certain metrics. For fractal curves such nonlinear energy forms can be obtained from a simple bare hands definition, see e.g. \cite{CaLa02}. 

Yet a different approach was taken in \cite{HPS04}, where the authors considered the Sierpinski gasket and, mimicking the construction of energy forms on post-critically finite self-similar sets, \cite{Ki01}, constructed $p$-energy forms on the gasket solving a renormalization problem. Related $p$-Laplacians were defined in \cite{StrW04}. A relatively simple scaling argument shows that these $p$-energies are quite different from those we defined above if $m$ is taken to be the Kusuoka measure, \cite{Ku89, Ka12, Ki08, Str06}, and $\mathcal{E}$ is the standard energy form on the gasket.
\end{remark}

\section{$L^p$-vector fields and reflexivity of Sobolev spaces}\label{S:bundle}

Let $(\mathcal{E},\mathcal{D}(\mathcal{E}))$ be a Dirichlet form on $L^2(X,m)$ so that Assumption \ref{A:basic} is satisfied. Following \cite{CS03} and \cite{Eb99} we discussed in \cite[Section 2]{HRT13} the existence of a Hilbert space $(\mathcal{H},\left\langle\cdot,\cdot\right\rangle_{\mathcal{H}})$ such that for all $a,b,c,d\in C_c(X)\cap \mathcal{D}(\mathcal{E}))$ we have $a\otimes b, c\otimes d\in \mathcal{H}$ and 
\[\left\langle a\otimes b, c\otimes d\right\rangle_\mathcal{H}=\int_Xbd\Gamma(a,c)dm.\]
By construction, finite linear combinations of such elements are dense in $\mathcal{H}$, and it is easy to see that also finite linear
combinations of elements $a\otimes b$ with $a,b\in\mathcal{A}$ are dense. Moreover, the operator 
\[\partial f:=f\otimes \mathbf{1},\quad  f\in\mathcal{A},\]
extends to a closed unbounded linear operator $\partial: L^2(X,m)\to\mathcal{H}$ with domain $\mathcal{D}(\mathcal{E})$,
and 
\[\left\|\partial f\right\|_{\mathcal{H}}^2=\mathcal{E}(f),\quad f\in\mathcal{D}(\mathcal{E}).\] 
An action $v\mapsto cv$ of $\mathcal{A}$ on $\mathcal{H}$ can be defined by linear continuation of $(a\otimes b)c:=a\otimes (bc)$, $a,b,c\in\mathcal{A}$, and density, it satisfies 
\[\left\|cv\right\|_{\mathcal{H}}\leq \left\|c\right\|_{\sup}\left\|v\right\|_{\mathcal{H}}\] 
for any $c\in\mathcal{A}$ and $v\in\mathcal{H}$. As a consequence, we observe the Leibniz rule $\partial(fg)=f\partial g + g\partial f$, $f,g\in\mathcal{A}$, note that $g\partial f=f\otimes g$ in $\mathcal{H}$. In one or the other form this construction appeared in many different contexts, see for instance \cite{CS03, CS09, Eb99, Gi15, HKT15, HRT13, HT15, IRT12, W00}, and for its probabilistic meaning, \cite[Section 5.6]{FOT94}, \cite[Section 9]{HRT13} and \cite{N85}. We refer to $\mathcal{H}$ as the \emph{space (or rather, module) of generalized $L^2$-vector fields}.

One can also provide a fiber-wise interpretation in a measurable sense. Recall that a collection $(\mathcal{H}_x)_{x\in X}$ of Hilbert spaces $(\mathcal{H}_x ,\left\langle\cdot,\cdot\right\rangle_{\mathcal{H}_x} )$ together with a subspace $\mathcal{M}$ of $\prod_{x\in X} \mathcal{H}_x$ is called a \emph{measurable field of Hilbert spaces} if
\begin{itemize}
\item[(i)] an element $\xi \in \prod_{x \in X} \mathcal{H}_x$, $\xi=(\xi_x)_{x\in X}$, is in $\mathcal{M}$ if and only if $x \mapsto\left\langle \xi_x, \eta_x\right\rangle_{\mathcal{H}_x}$ is measurable for any $\eta \in \mathcal{M}$,
\item[(ii)] there exists a countable set $\left\lbrace \xi^{(i)} : i \in \mathbbm{N}\right\rbrace \subset \mathcal{M}$ such that for all $x \in X$ the span of
$\lbrace \xi^{(i)}_x : i \in \mathbbm{N}\rbrace$ is dense in $\mathcal{H}_x$.
\end{itemize}

The elements $v=(v_x)_{x \in X}$ of $\mathcal{M}$ are usually referred to as \emph{measurable sections}. See for instance \cite[Section IV.8]{Tak02}. 

An observation already made in \cite{Eb99}, is that there are a measurable field $(\mathcal{H}_x)_{x\in X}$ of Hilbert spaces (or rather, modules) $\mathcal{H}_x$ on which the action of $a\in\mathcal{A}$ on $\omega_x\in\mathcal{H}_x$ is given by $a(x)\omega_x\in\mathcal{H}_x$ and such that the direct integral $\int_X^\oplus \mathcal{H}_x\:m(dx)$ is isometrically isomorphic to $\mathcal{H}$. In particular,
\[\left\langle u,v\right\rangle_{\mathcal{H}}=\int_X^\oplus \left\langle u_x,v_x\right\rangle_{\mathcal{H}_x}\:m(dx)\]
for all $u,v\in\mathcal{H}$, where, as above, for any $x\in X$ the symbol $v_x$ denotes the image of the associated projection $v\mapsto v_x$ from $\mathcal{H}$ into $\mathcal{H}_x$. Given $f,g\in\mathcal{F}$, we have $\Gamma(f,g)(x)=\left\langle \partial_xf,\partial_xg\right\rangle_{\mathcal{H}_x}$ for $m$-a.e. $x\in X$, where $\partial_xf:=(\partial f)_x$. See \cite[Section 2]{HRT13} for a proof. The spaces $\mathcal{H}_x$ may be viewed as substitutes for tangent spaces, see for instance \cite{HT15}. The direct integral is also denoted by $L^2(X,m,(\mathcal{H}_x)_{x\in X})$, because it is the space of (equivalence classes) of square integrable measurable sections.

\begin{remark}
In contrast to Riemannian manifolds the 'tangent spaces' $\mathcal{H}_x$ do not vary smoothly, but only measurably. Also, their dimension can change from one base point $x$ to another. Under the additional assumption that $m$ is minimal in a suitable way, the dimensions of the spaces $\mathcal{H}_x$ are a well-studied and useful quantity referred to as \emph{pointwise index} or \emph{Kusuoka-Hino index} of $(\mathcal{E},\mathcal{D}(\mathcal{E}))$, their essential supremum is called the \emph{martingale dimension}. See \cite{Hino08, Hino10, Hino13} and also \cite{BK16}.
For energy forms on self-similar fractals the martingale dimension is known to be bounded (by the spectral dimension) \cite{Hino13}, for p.c.f. self-similar fractals it is known to be one, \cite{Hino08}.
\end{remark}

As sketched in \cite[Section 6]{HRT13} one can also define spaces of $p$-integrable sections. For a measurable section $v=(v_x)_{x \in X}$ let 
\[ \Vert v \Vert_{L^p(X, m, (\mathcal{H}_x)_{x \in X})}:= \left( \int_X \Vert v_x \Vert_{\mathcal{H}_x}^p m(dx) \right)^{\frac{1}{p}}, \qquad 1\leq p<\infty,\]
and define the spaces $L^p(X,m,(\mathcal{H}_x)_{x \in X})$ as the collections of the respective equivalence classes of $m$-a.e. equal sections having finite norm. By a variant of the classical pointwise Riesz-Fischer argument they are seen to be separable Banach spaces. For $f \in \mathcal{A}$ and $v=(v_x)_{x \in X} \in L^p(X,m,(\mathcal{H}_x)_{x \in X})$ the product $fv$ is defined in the pointwise sense as the measurable section $x \mapsto f(x) v_x$. Since 
\[ \Vert fv \Vert_{L^p(X,m,(\mathcal{H}_x)_{x \in X})} \leq \Vert f(x)\Vert_{L^\infty(X,m)} \Vert v \Vert_{L^p(X,m,(\mathcal{H}_x)_{x \in X})} \]
the action $v \mapsto fv$ of $\mathcal{A}$ on $L^p(X,m,(\mathcal{H}_x)_{x \in X})$ is bounded. To the space $L^p(X,m,(\mathcal{H}_x)_{x\in X})$ we refer as the \emph{space of generalized $L^p$-vector fields}. Note that for any $1\leq p<+\infty$ we have 
\[\mathcal{E}^{(p)}(f)=\int_X \left\|\partial_xf\right\|_{\mathcal{H}_x}^p m(dx),\quad f\in \mathcal{A}.\]

The next fact was noted in \cite[Lemma 4.3]{CG03} for continuous fields of Hilbert spaces.
\begin{proposition}\label{P:reflexive}
The spaces $L^p(X,m,(\mathcal{H}_x)_{x\in X})$, $1<p<+\infty$, are uniformly convex and in particular, reflexive. For each $1<p<+\infty$ the spaces $L^p(X,m,(\mathcal{H}_x)_{x\in X})$ and $L^q(X,m,(\mathcal{H}_x)_{x\in X})$ with $1=1/p+1/q$ are the dual of each other.
\end{proposition}
We comment on a proof in Appendix \ref{S:uniconvex}. Proposition \ref{P:reflexive} implies the following useful fact.
\begin{corollary}\label{C:reflexive}
The Sobolev spaces $H_0^{1,p}(X,m)$ are separable for $1\leq p<+\infty$ and reflexive for $1<p<+\infty$.
\end{corollary}
Corollary \ref{C:reflexive} can be seen using the following well-known standard trick, see for instance \cite[Proposition 8.1 and 9.1]{Br11}. 
\begin{proof}
Since Cartesian products of reflexive spaces are reflexive, $L^p(X,m)\times L^p(X,m,(\mathcal{H}_x)_{x\in X})$ is reflexive for $1<p<+\infty$. The operator $T:H^{1,p}_0(X,m)\to L^p(X,m)\times L^p(X,m,(\mathcal{H}_x)_{x\in X})$, $Tf:=(f,\partial f)$ is an isometry from $H^{1,p}_0(X,m)$ onto the closed subspace $T(H^{1,p}_0(X,m))$ of $L^p(X,m)\times L^p(X,m,(\mathcal{H}_x)_{x\in X})$, and closed subspaces of reflexive spaces are reflexive. Therefore $T(H^{1,p}_0(X,m))$ is reflexive and consequently also $H^{1,p}_0(X,m)$. For $1\leq p<+\infty$ separability follows similarly, because it is stable under products and inherited to subsets.
\end{proof}

\begin{remark}
Although in the present setup the reflexivity of the spaces $H_0^{1,p}(X,m)$ may seem rather trivial to the reader, we would like to point out that in other approaches to Sobolev spaces on metric measure spaces it is a serious issue and may fail to hold, for some comments see \cite[p. 204]{HKShT15}.
\end{remark}

\section{Some examples}\label{S:Ex1}

\subsection{Classical Dirichlet integral} If $X=\Omega\subset \mathbb{R}^n$ is a bounded domain, then the classical Dirichlet integral $\mathcal{E}(f)=\int_\Omega |\nabla f|^2 dx$ with domain $W^{1,2}_0(\Omega)$, defined as the closure of $C_c^\infty(\Omega)$ in $W^{1,2}(\Omega)$, is a Dirichlet form on $L^2(\Omega)$ satisfying Assumption \ref{A:basic} and \ref{A:AL} with $\mathcal{A}_{\mathcal{L}}=\mathcal{A}=C_c^\infty(\Omega)$. The corresponding diffusion process is the Brownian motion in $\Omega$ killed at the boundary $\partial\Omega$, inside $\Omega$ this motion is isotropic. For any $1\leq p<+\infty$ the closure in $L^p(\Omega)$ of the functional
\[\mathcal{E}^{(p)}(f)=\int_\Omega |\nabla f|^p dx, \quad f\in C_c^\infty(\Omega),\]
has the domain $H_0^{1,p}(\Omega)=W^{1,p}_0(\Omega)$, \cite{AF03, Ma11}. Clearly $\Gamma(f)=|\nabla f|^2dx$ for all $f\in C_c^\infty(\Omega)$. Moreover, $\partial$ coincides with the usual gradient operator $\nabla$. For a.e. $x\in X$ the space $\mathcal{H}_x$ is isometrically isomorphic to $\mathbb{R}^n$, and $L^p(X,m,(\mathcal{H}_x)_{x\in X})$ is $L^p(\Omega, \mathbb{R}^n)$, up to isometry. The space $\mathcal{A}_{\mathcal{L}}\otimes\mathcal{A}$ is the span of vector fields of form $g\nabla f$, with $f,g\in C_c^\infty(\Omega)$.

\subsection{Degenerate forms} Let $X=(-1,1)^2\subset \mathbb{R}^2$ and consider the quadratic form
\[\mathcal{E}(f)=\int_{-1}^1 \int_{-1}^1 \left(\frac{\partial f}{\partial x_1}\right)^2dx_1dx_2+\int_{-1}^1\int_0^1 x_2 \left(\frac{\partial f}{\partial x_2}\right)^2dx_1dx_2, \quad f\in C_c^\infty((-1,1)^2).\] 
Since obviously $\frac{\partial}{\partial x_i}(x_2\vee 0)\in L^2((-1,1)^2)$, $i=1,2$, the form is closable in $L^2((-1,1)^2)$, \cite[Section 3.1, ($1^{\circ}$.a)]{FOT94}, and its closure satisfies Assumptions \ref{A:basic} and \ref{A:AL} with $m$ being the two-dimensional Lebesgue measure, $dm=dx_1dx_2$ and $\mathcal{A}_{\mathcal{L}}=\mathcal{A}=C_c^\infty(\Omega)$. We have 
\[\Gamma(f)(x_1,x_2)=\left(\frac{\partial f}{\partial x_1}(x_1,x_2)\right)^2+(x_2\vee 0)^2\left(\frac{\partial f}{\partial x_2}(x_1,x_2)\right)^2.\] For a.e. $x=(x_1,x_2)\in (-1,1)\times (-1,0)$ the spaces $\mathcal{H}_x$ are one-dimensional and for a.e. $x\in (-1,1)\times (0,1)$ two-dimensional. Roughly speaking, this means that in the lower half of the square the diffusion can move only in $x_1$-direction, while in the upper half it can also move in $x_2$-direction. The associated Sobolev spaces $H_0^{1,p}(X,m)$ inherit this degeneracy. 

\subsection{Superpositions} We revisit a special case of \cite[Example 2.3]{Hino13}. Again let $X=(-1,1)^2\subset \mathbb{R}^2$, we write $x=(x_1,x_2)$ for its elements. Now consider 
\[\mathcal{E}(f)=\int_{-1}^1\int_{-1}^1 |\nabla f(x_1,x_2)|^2dx_1dx_2+\int_{-1}^1 \left(\frac{\partial f}{\partial x_1}(x_1,0)\right)^2dx_1, \quad f\in C_c^\infty((-1,1)^2).\]
This form is closable in $L^2((-1,1)^2,m)$ with $dm=dx_1dx_2+dx_1\times\delta_0(dx_2)$, where $\delta_0$ is the Dirac measure at $0\in (-1,1)$, \cite[Section 3.1 ($2^{\circ}$), p.103]{FOT94}, and clearly $m$ is energy dominant.
Now 
\[\Gamma(f)=|\nabla f|^2+\left(\frac{\partial f}{\partial x_1}\right)^2.\] 
There is an $m$-null set outside of which  we have $\dim \mathcal{H}_x=2$ if $x_2\neq 0$ and $\dim \mathcal{H}_x=1$ if $x_2=0$. Again both Assumptions \ref{A:basic} and \ref{A:AL} are satisfied with $\mathcal{A}_{\mathcal{L}}=\mathcal{A}=C_c^\infty((-1,1)^2)$.

\subsection{Sierpinski gasket} 
Let $X$ be the classical Sierpinski gasket $K$ and $(\mathcal{E},\mathcal{D}(\mathcal{E}))$ its standard energy form, see for instance \cite{Str06}. We consider it in $L^2(K,\nu)$ where $m=\nu$ is the Kusuoka measure. The latter is defined as the sum $\nu=\nu_{h_1}+\nu_{h_2}$ of the energy measures of $h_1$ and $h_2$, where $\left\lbrace h_1, h_2\right\rbrace$ is an energy orthonormal system of non-constant harmonic functions on $K$. See for instance \cite{Ka12, Ki08, Ku89, T08}. Assumptions \ref{A:basic} and \ref{A:AL} are satisfied: The algebra $C^1(K)$ of functions of type $f=F(h_1,h_2)$ with $F\in C^1(\mathbb{R}^2)$ is dense in $\mathcal{D}(\mathcal{E})$ and by the chain rule for energy measures, can be taken as the algebra $\mathcal{A}$. In fact, we have $\Gamma(f)(x)=\left\langle Z_x\nabla F(y),\nabla F(y)\right\rangle_{\mathbb{R}^2}$, where $y(x):=(h_1(x),h_2(x))$, $x\in K$, and $Z=(Z_x)_{x\in X}$ is a measurable $(2\times 2)$-matrix valued function on $x$ such that $\rank Z_x=1$ for $\nu$-a.e. $x\in K$. The map $y$ is a homeomorphism $y:K\to y(K)$ of the compact (in Euclidean or resistance topology) space $K$ onto its image $y(K)$ in $\mathbb{R}^2$. Consequently $\Gamma(f)\in L^\infty(K,\nu)$. The density in the continuous functions follows from Stone-Weierstrass. Similarly we can use the space $C^2(K)$ of functions $f=F(h_1,h_2)$ with $F\in C^2(\mathbb{R}^2)$ as $\mathcal{A}_{\mathcal{L}}$, note that $\mathcal{L}f(x)=\tr(Z_xD^2F(y))$, where $D^2F$ denotes the Hessian of $F$, and clearly this is in $L^\infty(K,\nu)$. This space is also $\mathcal{E}$-dense in $\mathcal{D}(\mathcal{E})$. For details see \cite[Theorem 8]{T08}. Note that the result on the rank of $Z$ dictates that for $\nu$-a.e. $x\in K$ the dimension of $\mathcal{H}_x$ is one. 

\subsection{Products of fractals}
For simplicity consider $X=K\times [0,1]$, where $K$ is the classical Sierpinski gasket. We endow $X$ with the product measure $dm:=d\nu\times dx$, where $\nu$ is the Kusuoka measure on $K$ and on $I$ we use the one-dimensional Lebesgue measure $dx$. Let $\mathcal{E}_K$ be the standard energy form on $K$ with domain $\mathcal{D}(\mathcal{E}_K)$ and let $\mathcal{E}_I(f):=\int_0^1 (f')^2dx$ be the Dirichlet integral on $(0,1)$ with domain $W_0^{1,2}(0,1)$. On $L^2(K\times I, m)$ one can consider the product Dirichlet form $(\mathcal{E},\mathcal{D}(\mathcal{E}))$ defined by
\begin{multline}
\mathcal{D}(\mathcal{E}):=\left\lbrace f\in L^2(K\times I, m): \text{for a.e. $x_2\in I$ we have $f(\cdot, x_2)\in\mathcal{D}(\mathcal{E}_K)$}\right.\notag\\
\left. \text{and for $\nu$-a.e. $x_1\in K$ we have $f(x_1,\cdot)\in W_0^{1,2}(0,1)$} \right\rbrace
\end{multline}
and
\[\mathcal{E}(f)=\int_I\mathcal{E}_K(f(\cdot, x_2))dx_2+\int_K\int_I\left(\frac{\partial f}{\partial x_2}(x_1,x_2)\right)^2dx_2\nu(dx_1),\]
see \cite[Chapter V]{BH91}, \cite{Str05} or \cite[Section 5.6]{Str06}. We have 
\[\Gamma(f)(x_1,x_2))=\Gamma_K(f(\cdot,x_2))(x_1)+\left(\frac{\partial f}{\partial x_2}(x_1,x_2)\right)^2,\]
and it is not difficult to see that for $m$-a.e. $x=(x_1,x_2)\in K\times I$ the spaces $\mathcal{H}_x$ equal (up to isometry) the products $\mathcal{H}_{K,x_1}\times\mathcal{H}_{I,x_2}$ of the individual fibers. In particular, they are two-dimensional $m$-a.e. The Dirichlet form $(\mathcal{E},\mathcal{D}(\mathcal{E}))$ is regular and local and satisfies Assumption \ref{A:basic} with $\mathcal{A}=C^1(K)\otimes C_c^\infty((0,1))$ (with obvious multiplication).

\section{Existence of minimizers for convex functionals}\label{S:Exmin}

In this section we formulate the direct method for an abstract setup. To do so we follow classical presentations as can for instance be found in \cite[Sections 3.2 and 3.4]{Da08} or in \cite[Chapter 4]{Jost98}. Let $(\mathcal{E},\mathcal{D}(\mathcal{E}))$ be a Dirichlet form satisfying Assumption \ref{A:basic}, and whenever $1<p<2$ also Assumption \ref{A:AL}.

We start by observing lower semicontinuity for integral functionals on $L^p(X,m, (\mathcal{H}_x)_{x \in X})$, \cite[Lemma 4.3.1.]{Jost98}
\begin{lemma}\label{L:semicontinuous}
Let $1\leq p<+\infty$ and let $f=(f_x)_{x\in X}$ be a family of mappings $f_x : \mathcal{H}_x \rightarrow \mathbbm{R}$, $x \in X$, such that
\begin{itemize}
\item[(i)] for every $v \in L^p(X,m,(\mathcal{H}_x)_{x \in X})$ the function $x \mapsto f_x(v_x)$ is Borel measurable,
\item[(ii)]  $f_x$ is lower semicontinuous for every $x \in X$,
\item[(iii)]  there are a function $a \in L^1(X,m)$ and constant $b>0$ such that 
\begin{align}\label{E:1}f_x(v_x) \geq -a(x)+ b \Vert v_x \Vert_{\mathcal{H}_x}^p\end{align} for $m$-a.e. $x \in X$ and all $v \in L^p(X,m,(\mathcal{H}_x)_{x \in X})$. 
\end{itemize}
Then \[\Phi(v) := \int_X f_x(v_x) m(dx) \] is a lower semicontinuous functional on $L^p(X,m, (\mathcal{H}_x)_{x \in X})$. 
\end{lemma}

\begin{proof}
For any $v\in L^p(X, m, (\mathcal{H}_x)_{x\in X})$ the integral $\Phi(v)$ is well-defined as an element of the extended real axis because of (i) and (\ref{E:1}). Suppose $(v_{n})_n$ converges to $v$ in $L^p(X, m, (\mathcal{H}_x)_{x\in X})$. Then we can find a subsequence, for convenience again denoted by $(v_{n})_n$, such that its norms $\left\|(v_n)_x\right\|_{\mathcal{H}_x}$ converge pointwise $m$-a.e. to $\left\|v_x\right\|_{\mathcal{H}_x}$. Since $f_x$ is lower semicontinuous, we have 
\[ f_x(v_x) - b\Vert v_x \Vert^p_{\mathcal{H}_x} \leq \liminf_{n \rightarrow \infty} \left(f_x ((v_n)_x) - b \Vert (v_n)_x \Vert^p_{\mathcal{H}_x} \right)\]
$m$-a.e. and using \eqref{E:1} and Fatou's lemma, 
\[ \int_X \left(f_x (v_x) - b \Vert v_x \Vert^p_{\mathcal{H}_x}\right) m(dx) \leq \liminf_{n \rightarrow \infty}\int_X \left(f_x ((v_n)_x) - b \Vert (v_n)_x \Vert^p_{\mathcal{H}_x}\right) m(dx).\]
Because $(v_n)_n$ converges to $v$ in $L^p(X, m, (\mathcal{H}_x)_{x \in X})$, we have
\[\int_X b \Vert v_x \Vert^p_{\mathcal{H}_x} m(dx) = \lim_{n \rightarrow \infty} \int_X b \Vert (v_n)_x \Vert^p_{\mathcal{H}_x} m(dx),\]
so that by the superadditivity of $\liminf$,
\[\int_X f_x(v_x)m(dx) \leq \liminf_{n\to\infty} \int_X f_x((v_n)_x)m(dx). \]
\end{proof}

Convexity is inherited from the integrand to the functional.
\begin{lemma}\label{L:convex_functional} 
Suppose in addition to the assumptions in Lemma \ref{L:semicontinuous} that $f_x$ is convex for every $x \in X$. Then the functional $\Phi$ is also convex.
\end{lemma}

\begin{proof}
Let $v, w\in L^p(X, m, (\mathcal{H}_x)_{x \in X})$ and $t \in \left[ 0,1 \right]$. The convexity of each $f_x$ implies
\[\int_X f_x \left(tv_x + (1-t)w_x \right) m(dx)\leq \int_X\left(tf_x (v_x) +(1-t) f_x w_x \right)m(dx).\]
\end{proof}

It is a well known general fact that by convexity we can pass from the strong to the weak topology. For a proof see for instance \cite[Lemma 4.2.2.]{Jost98}
\begin{lemma}\label{L:wlsc}
Let $V$ be a convex subset of a separable reflexive Banach space, $F: V \rightarrow\bar{\mathbbm{R}}$ convex and lower semicontinuous. Then $F$ is also lower semicontinuous w.r.t. weak convergence.
\end{lemma}

We are interested in minimizing the convex functional  
\[ I\left[u\right]  := \int_X f_x(\partial_x u) m(dx). \]

The following is a version of a well known existence result, see e.g. \cite[Theorem 4.3.1]{Jost98}, adapted to our situation. Given an open set $\Omega$ and a function $g\in H^{1,p}_0(X,m)$ we write $g + H^{1,p}_0(\Omega,m)$ for the collection of all elements of $H^{1,p}_0(X,m)$ of form $g+\varphi$ with $\varphi\in H^{1,p}_0(\Omega,m)$. This encodes a \emph{generalized Dirichlet boundary condition}.

\begin{theorem}\label{T:exmin}
Let $1<p<+\infty$ let $\Omega\subset X$ be an open set and assume that the Poincar\'e inequality 
\[\Vert u \Vert_{L^p(\Omega, m)}^p \leq c\:\mathcal{E}^{(p)}(u),\quad u \in H^{1,p}_0(\Omega,m),\] holds, where $c>0$ is constant depending only on $\Omega$ and $p$. Let $f=(f_x)_{x\in X}$ be a family of mappings $f_x : \mathcal{H}_x \rightarrow \mathbbm{R}$, $x \in X$ such that
\begin{itemize}
\item[(i)] for every $v \in L^p(X,m,(\mathcal{H}_x)_{x \in X})$ the function $x \mapsto f_x(v_x)$ is Borel measurable,
\item[(ii)] the function $f_x$ is lower semicontinuous and convex for all $x\in X$,
\item[(iii)] there are a function $a \in L^1(X,m)$ and constant $b>0$ is satisfied such that  
\begin{align}\label{E:2}f_x(v_x) \geq -a(x)+ b \Vert v_x \Vert_{\mathcal{H}_x}^p\end{align} for almost all $x \in X$ and all $v\in L^p(X,m,(\mathcal{H}_x)_{x \in X})$. 
\end{itemize}
Then for any $g \in H^{1,p}_0(X,m)$ the functional
\[ I\left[u\right] = \int_X f_x ( \partial_x u ) m (dx) \] 
admits its infimum on $g + H^{1,p}_0(\Omega,m)$, i.e. there exists $u_0  \in g + H^{1,p}_0(\Omega,m)$ with \[I\left[u_0\right] = \inf_{u \in g+H^{1,p}_0(\Omega,m)}I \left[ u \right].\]
\end{theorem}

\begin{proof}
By Lemmas \ref{L:semicontinuous} and \ref{L:convex_functional} the functional $I$ is weakly lower semicontinuous on $H^{1,p}_0(X,m)$. Since $H^{1,p}_0(X,m)$ is separable and reflexive (Corollary \ref{C:reflexive}), 
$I$ is weakly lower semicontinuous on $H^{1,p}_0(X,m)$ by Lemma \ref{L:wlsc}. 

Let $(u_n)_n$ be a minimizing sequence in $g+H^{1,p}_0(\Omega,m)$, i.e. such that  
\[\lim_{n \rightarrow \infty}I[u_n] = \inf_{u \in g + H^{1,p}_0(\Omega,m)}I[u]. \]
From \eqref{E:2} we obtain 
\[ \int_X \left\| \partial_x u_n \right\|^p_{\mathcal{H}_x} m(dx) \leq \frac{1}{b} I\left[u_n\right] + \frac{1}{b} \int_X a(x) m(dx).\]
This implies that $(u_n)_n$ is bounded in $H^{1,p}_0(X,m)$. By Corollary \ref{C:reflexive} 
together with the theorems of Banach-Alaoglu and Eberlein-\v{S}mulian we can find a subsequence, which we will again denote by $(u_n)_n$, that converges weakly in $H^{1,p}_0(X,m)$ to some limit $u_0$. Since $g+ H^{1,p}_0(\Omega,m)$ is convex and closed, it is weakly closed, so that $u_0\in H^{1,p}_0(\Omega,m)$.

Combined with the weakly lower semicontinuity of $I$, this implies 
\[ I\left[u_0\right] \leq \liminf_{n \rightarrow \infty} I\left[u_n\right] = \lim_{n \rightarrow \infty} I\left[u_n\right] = \inf_{u \in g+ H^{1,p}_0(\Omega,m)} I\left [ u \right] \leq I\left[u_0\right].\]
and since $u_0 \in g + H^{1,p}_0(\Omega,m)$, we must have equality. 
\end{proof}

\section{Some examples}\label{E:Ex2}

\subsection{$p$-Dirichlet problems}

We sketch a simple application. Let $1<p<+\infty$, if $1<p<2$ let Assumption \ref{A:AL} be in force. Suppose $\Omega\subset X$ is a bounded open set and $g\in H_0^{1,p}(X,m)$. Let $u$ be a minimizer of the $p$-energy 
\[\mathcal{E}^{(p)}(u)=\int_X\left\|\partial_xu\right\|_{\mathcal{H}_x}^pm(dx)\] 
in $g+H_0^{1,p}(\Omega,m)$, which exists by Theorem \ref{T:exmin}, applied with $f_x(v):=\left\|v\right\|_{\mathcal{H}_x}^p$, $v\in\mathcal{H}_x$.
For any $\varphi\in\mathcal{A}$ and any $t\in\mathbb{R}$ we have 
\[\mathcal{E}^{(p)}(u+t\varphi)=\int_X\left\|\partial_xu+t\partial_x\varphi\right\|_{\mathcal{H}_x}^p m(dx),\]
so that 
\[\frac{d}{dt}\mathcal{E}^{(p)}(u+t\varphi)|_{t=0}=p\int_X\left\|\partial_xu\right\|_{\mathcal{H}_x}^{p-2}\left\langle \partial_xu,\partial_x\varphi\right\rangle_{\mathcal{H}_x}m(dx),\]
and by the minimality of $u$ we arrive at the weak form of the Euler-Lagrange equation,
\[\mathcal{E}^{(p)}(u,\varphi):=\int_X\left\langle \left\|\partial_x u\right\|_{\mathcal{H}_x}^{p-2}\partial_xu,\partial_x\varphi\right\rangle_{\mathcal{H}_x}m(dx)=0, \quad \varphi\in\mathcal{A}.\] 
Therefore $u$ may be regarded as a \emph{weak solution to the Dirichlet problem for the $p$-Laplacian} in $\Omega$ with boundary condition $g$ on $X\setminus \Omega$. Recall that for $u \in H^{1,p}_0(\Omega)$ the $p$-Laplacian can be defined in the variational sense by $\Delta_p u(v) := -\mathcal{E}^{(p)}(u,v)$, $v \in H^{1,p}_0(\Omega)$, see for instance \cite[Remark 6.1 and Examples 8.1]{HRT13}.

\subsection{Anisotropic functionals} Let $1<p<+\infty$, if $1<p<2$ let Assumption \ref{A:AL} be in force. Suppose that for $m$-a.e. $x\in X$ the space $\mathcal{H}_x$ is two-dimensional, as for instance in the last example in Section \ref{S:Ex1}. Let $\eta^{(1)},\eta^{(2)}\in\mathcal{H}$ be such that for any $x\in X$ with $\dim \mathcal{H}_x=2$, $\left\lbrace \eta^{(1)}_x,\eta^{(2)}_x\right\rbrace$ is an orthonormal basis in $\mathcal{H}_x$, see for instance \cite[Lemma 8.12]{Tak02}. By Theorem \ref{T:exmin} we can find a minimizer in $g+H_0^{1,p}(\Omega,m)$ for the functional $I$ with integrand defined by
\[f_x(v)=\left\|v\right\|_{\mathcal{H}_x}^p+|\left\langle v,\eta^{(1)}_x\right\rangle_{\mathcal{H}_x}|^p,\quad v\in\mathcal{H}_x,\]
if $\mathcal{H}_x$ is two-dimensional and by $f_x\equiv 0$ otherwise. This anisotropic functional could not be expressed 
in terms of the carr\'e operator $u\mapsto \Gamma(u)$ only.

\section{Constrained minimization problems}\label{S:constrained}

We translate some problems with integral constraints, \cite[Section 8]{Ev10}, to our setup. 

\subsection{Nonlinear Poisson equation}
Let $1< p <\infty$, if $1<p<2$ let Assumption \ref{A:AL} be satisfied. Suppose that $\Omega\subset X$ is open and $g\in H^{1,p}_0(X,m)$. We wish to minimize the energy functional 
\[I\left[w\right]:= \int_X \left\|\partial_x w \right\|_{\mathcal{H}_x}^p m(dx)\]
in the class $g+H_0^{1,p}(\Omega,m)$, but now subject to the additional condition that 
\[ J \left[ w \right] := \int_X G(w(x))m(dx) = 0, \]
where $G: \mathbb{R}\rightarrow \mathbb{R}$ is a given smooth function such that $|G'(z)| \leq C \left( \vert z \vert^{p-1} +1 \right)$ for some constant $C$. We introduce the admissible class 
\[ \mathfrak{A} := \lbrace w \in g+H^{1,p}_0(\Omega,m) \mid J \left[w \right]=0 \rbrace.\] 

\begin{theorem}
Assume the admissible class $\mathfrak{A}$ is nonempty. Then there exists $u \in \mathfrak{A}$ satisfying $I\left[ u \right]=\min_{w \in \mathfrak{A}}I\left[w\right]$. 
\end{theorem}

\begin{proof}
Since $\mathfrak{A}$ is convex we can, similarly as before, find a sequence $(u_n)_n \subset \mathfrak{A}$ with $\lim_n I\left[ u_n \right] = \inf_{w \in \mathfrak{A}} I \left[ w \right]$ that converges weakly to $u$ in $g+H^{1,p}_0(\Omega,m)$ with $I\left[u \right] \leq \inf_{w \in \mathfrak{A}} I \left[ w \right]$, and in particular, in $L^p(X,m)$, so that
\begin{multline}
\vert J(u) \vert = \vert J(u) - J( u_n) \vert \leq \int_X \vert G(u(x))-G((u_n(x))\vert m(dx)\notag\\ 
\leq C \int_X | u(x) - u_n(x)| \left(1 + |u(x)|^{p-1} + |u_n(x)|^{p-1} \right)m(dx). 
\end{multline}
The right hand side converges to zero, proving $J(u)=0$, hence $u\in\mathfrak{A}$.
\end{proof}

We turn to the corresponding Euler-Lagrange equation. 
\begin{theorem}
Let $\mathfrak{A}_0:= \lbrace w \in H^{1,p}_0(\Omega,m) \mid J \left[w \right]=0 \rbrace$. Suppose there exists $u\in \mathfrak{A}_0$ such that $I\left[u\right]=\min_{w \in \mathfrak{A}_0} I \left[ w \right]$.
Then we can find a real number $\lambda$ such that 
\begin{align}\label{E:lagrange_multiplier}\int_X \left\| \partial_x u \right\|_{\mathcal{H}_x}^{p-2} \left\langle \partial_x u, \partial_x v\right\rangle_{\mathcal{H}_x} m(dx) = \lambda \int_X G'(u(x)) v(x) m(dx) 
\end{align} for all $v \in H^{1,p}_0(\Omega,m)$. 
\end{theorem}

The number $\lambda$ is the \emph{Lagrange multiplier} corresponding to the integral constraint $J \left[ u \right]=0$.
The function $u$ as in the Theorem is a \emph{weak solution of the nonlinear Poisson equation} $-\Delta_p u = \lambda G'(u)$ in $\Omega$ with zero Dirichlet boundary condition on $X\setminus \Omega$ for the $p$-Laplacian $\Delta_p$.
In the case $p=2$ this is a nonlinear eigenvalue problem, see Section 8.4.1 in \cite{Ev10}.

To see the last theorem one can follow the proof of Theorem 2 in Section 8.4.1 of \cite{Ev10}, as Lagrange multiplier $\lambda$ one has to choose
\[ \lambda:= \frac{\int_X \left\| \partial_x u \right\|_{\mathcal{H}_x}^{p-2} \left\langle \partial_x u, \partial_x w\right\rangle_{\mathcal{H}_x} m(dx)}{\int_X G'(u(x))w(x) m(dx)}.\]

\subsection{Variational inequality}
In this subsection we assume $p=2$ and discuss variational problems with one-sided constraints. 
Let $\Omega\subset X$ be open. We are interested in minimizing the energy functional 
\[I \left[ w \right]:= \int_X \left\| \partial_x w \right\|_{\mathcal{H}_x}^2 -f(x)w(x) m(dx)\] 
among all functions $w$ belonging to the admissible class 
\[\mathfrak{A}:= \lbrace w \in g+H^{1,2}_0(\Omega,m) \mid w \geq h \text{ a.e. in }\Omega\rbrace,\]
where $f\in L^1(X,m)$, $f\not\equiv 0$ and $h\in H_0^{1,2}(X,m)$. The function $h$ is called the \emph{obstacle}. We revisit a well known existence and uniqueness result, see \cite[Section 8.4.2]{Ev10}

\begin{theorem}
Assume the admissible set $\mathfrak{A}$ is nonempty. Then there exists a unique function $u \in  \mathfrak{A}$ satisfying $I \left[u \right] = \min_{w \in \mathfrak{A}} I \left[w \right]$.
\end{theorem}

\begin{proof}
The existence of a minimizer follows as before. To see uniqueness, we assume $u, \tilde{u} \in \mathfrak{A}$ are  minimizers and $u\neq \tilde{u}$. Then $w := \frac{1}{2}\left( u + \tilde{u}\right) \in \mathfrak{A}$, and 
\begin{align*}
I \left[ w \right] &= \int_X \frac{1}{4} \left\| \partial_x u + \partial_x \tilde{u}\right\|_{\mathcal{H}_x}^2 - \frac12 f(x) (u(x)+\tilde{u}(x)) m(dx)\\
&= \int_X \frac{1}{4} \left(2 \left\| \partial_x u \right\|_{\mathcal{H}_x}^2 + 2 \left\| \partial_x \tilde{u}\right\|_{\mathcal{H}_x}^2- \left\| \partial_x u - \partial_x  \tilde{u} \right\|_{\mathcal{H}_x}^2 \right)-\frac12 f(x) (u(x)+\tilde{u}(x)) m(dx) \\
 &< \frac{1}{2}\int_X \left\| \partial_x u \right\|_{\mathcal{H}_x}^2- f(x)u(x) m(dx)+\frac{1}{2} \int_X  \left\| \partial_x \tilde{u} \right\|_{\mathcal{H}_x}^2 -f(x) \tilde{u}(x) m(dx)\notag\\
 & = \frac{1}{2}I\left[ u \right] +\frac{1}{2}I\left[ \tilde{u} \right]. \end{align*}
The strict inequality follows because $u \neq \tilde{u}$ and because $I[\mathbf{1}]=\int_Xf(x)m(dx)$, so that the functional cannot produce the same value for two functions that differ by a constant. The above inequality contradicts the minimality of $u$ and $\tilde{u}$. 
\end{proof}

For the present problem the Euler-Lagrange equation is replaced by an inequality.

\begin{theorem}
Let $u \in \mathfrak{A}$ be the unique solution of $I\left[ u \right] = \min_{w \in \mathfrak{A}} I \left[ w \right]$. 
Then 
\[\int_X \left\langle \partial_x u, \partial_x(w-u)\right\rangle_{\mathcal{H}_x} m(dx) \geq \int_X f(x)(w(x)-u(x))m(dx)\]
for all $w \in \mathfrak{A}$. 
\end{theorem}

\begin{proof}
Fix any element $w \in \mathfrak{A}$. The convexity of $\mathfrak{A}$ implies that for any $\tau \in [0,1]$ the function $u+\tau ( w-u) = (1-\tau ) u + \tau w $ is an element of $\mathfrak{A}$. Consequently, if  we set $i(\tau) : = I\left[ u + \tau (w-u) \right]$, we see that $ i(0) \leq i( \tau ) $ for all $\tau\in [0,1]$. Hence $i'(0)\geq 0$. Now if $\tau\in (0,1]$, then 
\[\frac{i(\tau) - i(0)}{\tau}= \int_X \left\langle \partial_x u, \partial_x(w-u)\right\rangle_{\mathcal{H}_x} + \frac{\tau}{2} \left\| \partial_x (w-u)\right\|^2 - f(x)(w(x)-u(x)) m(dx),\]
and taking the limit $\tau\to 0$, we obtain the result.
\end{proof}

Problems of this type occur for instance in elastic plastic torsion problems, \cite{Ti66, Ti67}. It would be interesting to see whether there are meaningful fractal analogs of such models.

\appendix

\section{Uniform convexity of $L^p$}\label{S:uniconvex}

To prove Proposition \ref{P:reflexive} one can follow \cite[Chapter Five, \S 26, Section 7]{K69}. For convenience we briefly revisit these arguments. We wish to point out that if one is interested in reflexivity only, one could give a slightly shorter proof (as discussed in the cited reference).  

Recall first that a normed space $(V, \left\|\cdot\right\|)$ is called \emph{uniformly convex} if for any $0<\varepsilon\leq 2$ there exists some $\delta>0$ such that for all $u,v\in V$ with $\left\|u\right\|\leq 1$, $\left\|v\right\|\leq 1$ and $\left\|u-v\right\|>\varepsilon$ we have $\left\|\frac12(u+v)\right\|\leq 1-\delta$. The condition for uniform convexity may be seen as the generalization of the parallelogram identity in (pre-) Hilbert spaces, from which it is immediate that any (pre-) Hilbert space is uniformly convex. By Milman's theorem, \cite[Chapter Five, \S 26, Section 6, (4)]{K69} every uniformly convex Banach space is reflexive. We also need the following inequalities due to Clarkson, \cite[Chapter Five, \S 26, Section 7, p. 357]{K69}.

\begin{proposition}\label{P:Clarkson}
Suppose $(V, \left\|\cdot\right\|)$ is a uniformly convex normed space and $1<p<\infty$. 
Given $\varepsilon >0$ there exists some $\delta>0$ such that for any $u,v\in V$ with $\left\|u\right\|\leq 1$, $\left\|v\right\|\leq 1$ and $\left\|u-v\right\|\geq \varepsilon$ we have 
\[\big\|\frac12(u+v)\big\|^p\leq (1-\delta)\left(\frac{\left\|u\right\|^p+\left\|v\right\|^p}{2}\right).\] 
For any $u,v\in V$ therefore  
\[\big\|\frac12(u+v)\big\|^p\leq \left(1-\delta\left(\frac{\left\|u-v\right\|}{\max(\left\|u\right\|,\left\|v\right\|)}\right)\right)\left(\frac{\left\|u\right\|^p+\left\|v\right\|^p}{2}\right).\]
\end{proposition}

We prove Proposition \ref{P:reflexive}. To shorten notation we will use the abbreviation $L^p$ for $L^p(X,m,(\mathcal{H}_x)_{x\in X})$. For simplicity we assume that $m$ is a probability measure.
\begin{proof}
Let $1<p<+\infty$. We verify the uniform convexity of $L^p$, the reflexivity follows. Suppose $0<\varepsilon \leq 2$, $u,v\in L^p$, $\left\|u\right\|_{L^p}\leq 1$, $\left\|v\right\|_{L^p}\leq 1$ and $\left\|u-v\right\|_{L^p}\geq \varepsilon$. In the following we work with fixed $m$-versions of $u$ and $v$, denoted by the same symbols; the result does not depend on their choice.
Let $M\subset X$ be the set of all $x\in X$ such that 
\begin{equation}\label{E:M}
\left\|u_x-v_x\right\|_{\mathcal{H}_x}^p\geq \frac{\varepsilon^p}{4}(\left\|u_x\right\|_{\mathcal{H}_x}^p+\left\|v_x\right\|_{\mathcal{H}_x}^p)\geq \frac{\varepsilon^p}{4}\max (\left\|u_x\right\|_{\mathcal{H}_x}^p, \left\|v_x\right\|_{\mathcal{H}_x}^p).
\end{equation}
Applying Proposition \ref{P:Clarkson} (ii) to the Hilbert space $\mathcal{H}_x$, we obtain 
\begin{equation}\label{E:firstpiece}
\big\|\frac12(u_x+v_x)\big\|_{\mathcal{H}_x}^p\leq \left(1-\delta\frac{\varepsilon}{4^{1/p}}\right)\left(\frac12(\left\|u_x\right\|_{\mathcal{H}_x}^p+\left\|v_x\right\|_{\mathcal{H}_x}^p)\right)
\end{equation}
for all $x\in M$. For $X\setminus M$ we have 
\[\int_{X\setminus M}\left\|u_x-v_x\right\|_{\mathcal{H}_x}m(dx)\leq \frac{\varepsilon^p}{4}\int_X (\left\|u_x\right\|_{\mathcal{H}_x}^p+\left\|v_x\right\|_{\mathcal{H}_x}^p)m(dx)\leq \frac{\varepsilon^p}{2},\]
so that, using the above assumptions on $u$ and $v$, 
\[\int_M \left\|u_x-v_x\right\|_{\mathcal{H}_x}m(dx)\geq \frac{\varepsilon^p}{2}.\]
Consequently $2\max (\left\|u|_M\right\|_{L^p}, \left\|v|_M\right\|_{L^p})\geq \left\|u|_M-v|_M\right\|_{L^p}\geq \varepsilon/2^{1/p}$, i.e.
\begin{equation}\label{E:secondpiece}
\max(\left\|u|_M\right\|_{L^p}^p, \left\|v|_M\right\|_{L^p}^p)\geq \frac{\varepsilon^p}{2^{p+1}}.
\end{equation}
Since by the elementary inequality $a^p+b^p\geq 2^{1-p}(a+b)^p$ for $a,b \geq 0$ the integrand is nonnegative, we have
\begin{multline}
\int_X\left(\frac12(\left\|u_x\right\|_{\mathcal{H}_x}^p+\left\|v_x\right\|_{\mathcal{H}_x}^p)-(\frac12\left\|u_x+v_x\right\|_{\mathcal{H}_x})^p\right)m(dx)\notag\\
\geq \int_M\left(\frac12(\left\|u_x\right\|_{\mathcal{H}_x}^p+\left\|v_x\right\|_{\mathcal{H}_x}^p)-(\frac12\left\|u_x+v_x\right\|_{\mathcal{H}_x})^p\right)m(dx).
\end{multline}
By (\ref{E:firstpiece}) this is greater or equal to 
\[\delta\frac{\varepsilon}{4^{1/p}}\frac12\int_M(\left\|u_x\right\|_{\mathcal{H}_x}^p+\left\|v_x\right\|_{\mathcal{H}_x}^p)m(dx)\geq \delta\frac{\varepsilon}{4^{1/p}}\frac{\varepsilon^p}{2^{p+2}}.\]
where we have used (\ref{E:secondpiece}). This implies
\[\left\|\frac12(u+v)\right\|_{L^p}\leq \left(1-\delta \frac{\varepsilon}{4^{1/p}}\frac{\varepsilon^p}{2^{p+2}}\right).\]
It remains to show that for any $1<p<+\infty$ the space $L^q$, $\frac{1}{p}+\frac{1}{q}=1$, is the dual of $L^p$. We repeat the classical arguments to point out that a Radon-Nikodym theorem is not needed. Given $v\in L^q$ consider the linear functional $v\mapsto \left\langle v,u\right\rangle=\int_X\left\langle u_x,v_x\right\rangle_{\mathcal{H}_x}m(dx)$, $u\in L^p$. By H\"older's inequality this is a member of $(L^p)'$, hence $L^q$ can be identified with a closed subspace of $(L^p)'$. We claim that $\sup_{\left\|u\right\|_{L^p}\leq 1}|\left\langle u,v\right\rangle|=\left\|v\right\|_{L^q}$. The inequality $\leq$ is clear from H\"older. Now define a measurable section $u=(u_x)_{x\in X}$ by $u_x:=\left\|v_x\right\|_{\mathcal{H}_x}^{q-2}v_x$ for $x\in \left\lbrace v\neq 0\right\rbrace$ and $u_x:=0$ for $x\in \left\lbrace v= 0\right\rbrace$. Then $\left\|u\right\|_{L^p}=\left\|v\right\|_{L^q}^q$, so that $u\in L^p$. Moreover, $\left\langle v,\frac{u}{\left\|u\right\|_{L^p}}\right\rangle=\left\|v\right\|_{L^q}$, proving the claim. Consequently on $L^q$ the norm of $(L^p)'$ coincides with the norm in $L^q$, and since $L^q$ is complete, it is a closed subspace of $(L^p)'$. If $L^q$ were a proper subspace we could find a nontrivial bounded linear functional on $(L^p)'$ vanishing on $L^q$. Since $L^p$ is reflexive, this functional must be given by some $u\in L^p$. But then $\left\langle u,v\right\rangle=0$ for all $v\in L^q$, and for $v\in L^q$ defined by $v=(v_x)_{x\in X}$ by $v_x:=\left\|u_x\right\|_{\mathcal{H}_x}^{q-2}u_x$ for $x\in \left\lbrace u\neq 0\right\rbrace$ and $v_x:=0$ for $x\in \left\lbrace u= 0\right\rbrace$ we obtain $u=0$ in $L^p$, a contradiction.
\end{proof}

\section{Closability of $p$-energies and distributional gradients}\label{S:closable}

In order to prove Theorem \ref{T:closable} we need some preparations. We start with a version of \cite[Lemma 7.2]{HRT13}. Because the proof is an inessential modification of the one given there, we omit it.
\begin{lemma}\label{L:dense}
The space $\mathcal{A}\otimes\mathcal{A}$ is dense in all $L^p(X,m,(\mathcal{H}_x)_{x\in X})$, $1<p<+\infty$. Under Assumption \ref{A:AL} also the space $\mathcal{A}_{\mathcal{L}}\otimes\mathcal{A}$ is dense in all $L^p(X,m,(\mathcal{H}_x)_{x\in X})$, $1<p<+\infty$.
\end{lemma}

We rely on an integration by parts formula which involves the divergence. Recall that the adjoint operator $\partial^\ast: \mathcal{H}\rightarrow L^2(X, m)$ of $\partial$ is defined by saying that $v \in \mathcal{H}$ is  a member of $\mathcal{D}(\partial^\ast)$ if there exists $v^\ast \in L^2(X, m)$ such that $\langle f,v \rangle_{L^2(X,m)}= \langle \partial f , v \rangle_{\mathcal{H}}$ for all $f \in \mathcal{D}(\mathcal{E})$. In this case $\partial^\ast v := v^\ast$ and \[\langle f, \partial^\ast v \rangle_{L^2(X, m)}=  \langle \partial f, v \rangle_{\mathcal{H}}, \quad f \in \mathcal{D}(\mathcal{E}).\] 
The operator $-\partial^\ast$ is a \emph{generalized divergence}. (Note that in \cite{HRT13} we used another sign convention.)
Alternatively $\partial^\ast v$ can be defined for any $v\in\mathcal{H}$ in a distributional sense by setting 
\[\partial^\ast v(\varphi):=\left\langle \partial \varphi, v\right\rangle_{\mathcal{H}},\quad \varphi\in\mathcal{A}.\]
For $v=g\partial f$ with $f,g\in\mathcal{A}$ we then have $\partial^\ast(g\partial f)(\varphi)=\int_Xg\Gamma(f,\varphi)dm$, $\varphi\in\mathcal{A}$, and therefore
\begin{equation}\label{E:simplebound}
|\partial^\ast(g\partial f)(\varphi)|\leq \left\|g\right\|_{\sup}\mathcal{E}(f)^{1/2}\mathcal{E}(\varphi)^{1/2},
\end{equation}
as pointed out in \cite[Section 3]{HRT13}. This is sufficient to prove closability for $p\geq 2$, \cite[Theorem 6.1]{HRT13}.

Now suppose Assumption \ref{A:AL} is in force. For $f\in\mathcal{A}_{\mathcal{L}}$ and $g\in\mathcal{A}$ we have, similarly as in \cite[Lemma 3.2]{HRT13}, the identity
\[\partial^\ast(g\partial f)(\varphi)=\mathcal{E}(g\varphi, f)-\int_X\varphi d\Gamma(f,g)dm,\quad \varphi\in\mathcal{A}.\]
By (\ref{E:generator}) this implies 
\begin{equation}\label{E:betterbound}
|\partial^\ast(g\partial f)(\varphi)|\leq (\left\|g\right\|_{L^\infty(X,m)}\left\|\mathcal{L}f\right\|_{L^\infty(X,m)}+\left\|\Gamma(f)\right\|_{L^\infty(X,m)}^{1/2}\left\|\Gamma(g)\right\|_{L^\infty(X,m)}^{1/2})\left\|\varphi\right\|_{L^1(X,m)}
\end{equation}
for all $\varphi\in\mathcal{A}$, and since $\mathcal{A}$ is dense in $L^1(X,m)$, the functional $\partial^\ast(g\partial f)$ and the estimate (\ref{E:betterbound}) extend to all $\varphi\in L^1(X,m)$. By linear extension we can therefore define $\partial^\ast v$ for any $v\in\mathcal{A}_{\mathcal{L}}\otimes\mathcal{A}$ as an element of $L^\infty(X,m)$. In particular, $\mathcal{A}_{\mathcal{L}}\otimes\mathcal{A}\subset \mathcal{D}(\partial^\ast)$.

This allows to define gradients $\partial f$ in a \emph{distributional sense} for all $f\in L^1(X,m)$. The space $\mathcal{A}_{\mathcal{L}}\otimes\mathcal{A}$ can be endowed with the norm $v\mapsto \left\|\partial^\ast v\right\|_{L^\infty(X,m)}+\left\|v\right\|_{\mathcal{H}}$. Now suppose $f\in L^1(X,m)$. Setting
\begin{equation}\label{E:distgrad}
\partial f(v):=\int_X f\, \partial^\ast v \, dm,\quad v\in\mathcal{A}_{\mathcal{L}}\otimes\mathcal{A},
\end{equation}
we observe
\[|\partial f(v)|\leq \left\|f\right\|_{L^1(X,m)}(\left\|\partial^\ast v\right\|_{L^\infty(X,m)}+\left\|v\right\|_{\mathcal{H}}),\]
so that $\partial f$ is is seen to be an element of the dual space $(\mathcal{A}_{\mathcal{L}}\otimes\mathcal{A})'$, and its norm in that space is bounded by $\left\|f\right\|_{L^1(X,m)}$.
 
For any $1<p<+\infty$ we can regard $L^p(X,m,(\mathcal{H}_x)_{x\in X})$ as a subset of $(\mathcal{A}_{\mathcal{L}}\otimes\mathcal{A})'$ by putting
\[\xi(v):=\left\langle v,\xi\right\rangle=\int_X\left\langle v_x,\xi_x\right\rangle_{\mathcal{H}_x}m(dx),\quad v\in\mathcal{A}_{\mathcal{L}}\otimes\mathcal{A},\]
for $\xi \in L^p(X,m,(\mathcal{H}_x)_{x\in X})$, note that by H\"older's inequality the norm of $\xi$ in $(\mathcal{A}_{\mathcal{L}}\otimes\mathcal{A})'$ is bounded by $\left\|\xi\right\|_{L^p(X,m,(\mathcal{H}_x)_{x\in X})}$. By the density of $\mathcal{A}_{\mathcal{L}}\otimes\mathcal{A}$ in $L^q(X,m,(\mathcal{H}_x)_{x\in X})$, where $\frac1p+\frac1q=1$, it then follows that if $w\in (\mathcal{A}_{\mathcal{L}}\otimes\mathcal{A})'$ is such that $w=\xi$ in this space, then $w\in L^p(X,m,(\mathcal{H}_x)_{x\in X})$ and $w=\xi$ in $L^p(X,m,(\mathcal{H}_x)_{x\in X})$. In particular, if $f\in \mathcal{A}$ then 
\begin{equation}\label{E:recovergradient}
\partial f(v)=\left\langle \partial f,v\right\rangle,\quad v\in \mathcal{A}_{\mathcal{L}}\otimes\mathcal{A},
\end{equation}
i.e. the distributional gradient of $f$ equals the gradient $\partial f\in L^p(X,m,(\mathcal{H}_x)_{x\in X})$.

We prove Theorem \ref{T:closable}.
\begin{proof}
Suppose $(f_n)_n\subset \mathcal{A}$ is $\mathcal{E}^{(p)}(\cdot)^{1/p}$-Cauchy and $\lim_n f_n=0$ in $L^p(X,m)$. Then $(\partial f_n)_n$ is Cauchy in $L^p(X,m,(\mathcal{H}_x)_{x\in X})$ and therefore converges to some limit $\xi$ in this space. It suffices to show $\xi=0$. 

If $2\leq p<+\infty$ we can proceed as in \cite[Theorem 6.1]{HRT13}: In this case the finiteness of $m$ implies that $(f_n)_n$ is $\mathcal{E}$-Cauchy, what by (\ref{E:simplebound}) shows that for any $f,g\in\mathcal{A}$ we have
\[\int_X\left\langle g\partial f, \xi\right\rangle_{\mathcal{H}_x}m(dx)=\lim_n\left\langle g\partial f,\partial f_n\right\rangle_{\mathcal{H}}=\lim_n\partial^\ast(g\partial f)(f_n)=0,\]
and by linear extension and Lemma \ref{L:dense} this implies that $\xi=0$. 

If $1<p<2$ and Assumption \ref{A:AL} holds, then for any $f\in\mathcal{A}_{\mathcal{L}}$ and $g\in\mathcal{A}$ we obtain
\[\int_X\left\langle g\partial f, \xi\right\rangle_{\mathcal{H}_x}m(dx)=\lim_n \left\langle g\partial f,\partial f_n\right\rangle=\lim_n\partial f_n(g\partial f)=\lim_n\int_X\partial^\ast(g\partial f) f_n\, m(dx)=0\]
by (\ref{E:distgrad}) and (\ref{E:recovergradient}) and because $\partial^\ast(g\partial f)\in L^\infty(X,m)\subset L^q(X,m)$. Again linear extension and Lemma \ref{L:dense} imply that $\xi=0$.
\end{proof}

The closability can also be stated in terms of the operator $\partial$.
\begin{corollary}
The linear operator $(\partial,\mathcal{A})$ is closable in $L^p(X,m)$,  $2\leq p<+\infty$. If Assumption \ref{A:AL} is satisfied, then it is also closable in $L^p(X,m)$, $1<p<2$.
\end{corollary}
This implies that under the respective hypoheses $\partial$ extends to a closed unbounded linear operator 
$\partial:L^p(X,m)\to L^p(X,m,(\mathcal{H}_x)_{x\in X})$ with domain $H_0^{1,p}(X,m)$.

\begin{remark}\mbox{}
\begin{enumerate}
\item[(i)] We wish to point out that, as in the case $p=2$, the closability of $(\mathcal{E}^{(p)},\mathcal{A})$ is equivalent to the lower semicontinuity of $\mathcal{E}^{(p)}$, seen as a functional on $L^p(X,m)$ taking values in $[0,+\infty]$. The proof that the existence of a closed extension implies lower semicontinuity uses Banach-Alaoglu (together with Eberlein-\v{S}mulian and Mazur's lemma) and the reflexivity of $H^{1,p}_0(X,m)$. 
\item[(ii)] Similarly as in the case $p=2$ closability and lower semicontinuity of $\mathcal{E}^{(p)}$ with respect to the norm in $L^p(X,m)$ are equivalent to the closability of $\mathcal{E}^{(p)}$ with respect to the supremum norm and also equivalent to the lower semicontinuity of $\mathcal{E}^{(p)}$ with respect to the supremum norm. This can be seen similarly as in \cite[Sections 6, 8 and 10]{H16}, see also \cite{HT15b}. This use of the supremum norm goes back to \cite{Mo95}. A detailed and very general discussion of closability and lower semicontinuity in $L^p$-spaces can be found in \cite{Schm17}.
\end{enumerate}
\end{remark}

\section{Sobolev spaces $W^{1,p}(X,m)$}

As a by-product of the above proof of closability for $1<p<2$ one can provide an analog of the most classical definition of Sobolev spaces $W^{1,p}(\Omega)$. Let Assumption \ref{A:AL} be in force. 

For any $1<p<+\infty$ set
\[W^{1,p}(X,m):=\left\lbrace f\in L^p(X,m): \partial f\in L^p(X,m)\right\rbrace\]
and consider this vector space with the norm
\[f\mapsto (\left\|f\right\|_{L^p(X,m)}+\left\|\partial f\right\|_{L^p(X,m,(\mathcal{H}_x)_{x\in X}})^{1/p}.\]
Now the classical proof shows that they are Banach spaces: If $(f_n)_n$ is Cauchy in $W^{1,p}(X,m)$ then there exist some $f\in L^p(X,m)$ such that $f=\lim_n f_n$ in $L^p(X,m)$ and some $\xi\in L^p(X,m,(\mathcal{H}_x)_{x\in X})$ such that $\lim_n \partial f_n=\xi$ in $L^p(X,m,(\mathcal{H}_x)_{x\in X})$. Since 
\[\xi(v)=\left\langle v,\xi\right\rangle=\lim_n\left\langle v,\partial f_n\right\rangle=\lim_n \partial f_n(v)=\lim_n \int_X \partial^\ast v\, f_n\, dm=\int \partial^\ast v\,f\, dm=\partial f(v)\]
for all $v\in \mathcal{A}_{\mathcal{L}}\otimes\mathcal{A}$, we have $\xi=\partial f\in L^p(X,m,(\mathcal{H}_x)_{x\in X})$, what shows that $\lim_n f_n=f$ in $W^{1,p}(X,m)$, as desired.

Obviously 
\[H_0^{1,p}(X,m)\subset W^{1,p}(X,m).\] 
Looking at the classical $p$-energies on bounded Euclidean domains as discussed in Section \ref{S:Ex1} shows that in general the converse inclusion will not hold: In this case $\partial f$ for $f\in L^1(\Omega)\subset L^1_{\loc}(\Omega)$ coincides with $\nabla f$, seen as a regular distribution on $\Omega$, and we have $W^{1,p}(X,m)=W^{1,p}(\Omega)$, which is strictly larger than $H_0^{1,p}(X,m)=W^{1,p}_0(\Omega)$.

\end{document}